\documentclass[12pt, reqno]{amsart}

\usepackage{etex}

\usepackage[colorlinks=true,linkcolor=blue,citecolor=blue]{hyperref}

\usepackage{amsmath,amsthm,amsfonts,amssymb,mathbbol}

\usepackage{dsfont, upgreek}
\usepackage{mathtools}
\usepackage{epsfig}

\usepackage[all]{xy}
\usepackage{tikz}
\usepackage{tikz-cd}
\usetikzlibrary{matrix,arrows,decorations.pathmorphing}
\usetikzlibrary{topaths}

\usepackage[export]{adjustbox}
\theoremstyle{plain}% default
\newtheorem{theorem}{Theorem}[section]

\newtheorem{lemma}[theorem]{Lemma}
\newtheorem{corollary}[theorem]{Corollary}

\theoremstyle{definition}
\newtheorem{definition}[theorem]{Definition}

\theoremstyle{remark}

\begin{document}

\title[Equivalence of higher $\rho$ invariants]{Equivalence of different definitions of higher $\rho$ invariants }

\author[H. Liu]{Hongzhi Liu}
\address{H. Liu, School of Mathematics, Shanghai University of Finance and Economics,  Shanghai, 200433, P. R. China.  e-mail: liu.hongzhi@sufe.edu.cn}
\author[Z. Xie]{Zhizhang Xie}
\address{Z. Xie, Department of Mathematics, Texas A\&M University, College Station, TX, 77843, USA.  e-mail: xie@tamu.edu}
\author[G. Yu]{Guoliang Yu}
\address{G. Yu, Department of Mathematics, Texas A\&M University, College Station, TX, 77843, USA.  e-mail: guoliangyu@math.tamu.edu}

\maketitle

\bibliographystyle{abbrv}

\begin{abstract}
For each orientation-preserving homotopy equivalence between two closed oriented smooth  manifolds, there are mainly two different approaches to the higher $\rho$ invariant associated to this homotopy equivalence. In this article, we show that these two definitions of the higher $\rho$ invariant are equivalent. \end{abstract}

\section{Introduction}\label{sec introduction}

The Fredhoml index of elliptic differential operators on closed smooth manifolds is a fundamental invariant in mathematics. When the underlying closed manifold has nontrivial fundamental group, there is a natural higher index, a vast generalization of the Fredholm index, that lies in the $K$-theory of the group $C^\ast$-algebra of the fundamental group. If the higher index of an elliptic operator vanishes, then any specific trivialization of the higher index induces a secondary invariant, called the higher $\rho$ invariant of this trivialization \cite{HR053MappingsurgerytoanalysisIIIExactsequences, Roeindexthecoarsegeoandtopofmfld}. The higher $\rho$ invariant is an obstruction to the inverse of an invertible elliptic differential operator being local \cite{chenliuyu}.  Higher $\rho$ invariants can be used to detect various geometric information of the underlying space. For example, let $M$ be a closed smooth spin manifold, and $\widetilde M$ be the universal covering space of $M$. If $M$ admits a Riemannian metric of positive scalar curvature, then it follows from the Lichnerowicz formula that the associated Dirac operator is invertible. In other words, the positivity of the scalar curvature  provides a trivialization of the higher index of the Dirac operator, hence determines a higher $\rho$ invariant of the associated Dirac operator. This particular higher $\rho$ invariant may be used to measure the size of the moduli space of positive scalar curvature metrics \cite{Xieyuhihgerrhoandmodulispace}. For a brief survey on the higher $\rho$ invariant and other applications, we refer the reader to \cite{XieYusurvey}.

In this paper, we shall mainly focus on higher $\rho$ invariants that arise from orientation-preserving homotopy equivalences of oriented closed manifolds. More precisely,  let $f: M\to N$ be an orientation preserving homotopy equivalence between two closed oriented smooth manifolds $M$ and $N$. Let $\widetilde M$ (resp. $\widetilde N$) be the  universal covering space of $M$ (resp. $N$), and $\Gamma$  the fundamental group of $M$ and $N$. Let $D_{\widetilde M}$ (resp. $D_{\widetilde N}$) be the signature operator on $\widetilde M$ (resp. $\widetilde N$). Then the higher index of the signature operator
\[D=
\begin{pmatrix}
D_{\widetilde N} & \\
 & -D_{\widetilde M}
\end{pmatrix}
\]
on the disjoint union $N\sqcup (-M)$ is trivial with the trivialization given by $f$. Hence one can define the higher $\rho$ invariant of  $D$ associated to $f$. This higher $\rho$ invariant detects whether the homotopy equivalence $f$ can be deformed into a homeomorphism. More precisely, if the higher $\rho$ invariant of $f$ is not zero, then $f$ cannot be deformed into a homeomorphism. 

% It has been used to transfer the surgery exact sequence to the exact sequence of the $K$-theory of geometric $C^*$-algebras (\cite{HR053MappingsurgerytoanalysisIIIExactsequences, WXY18Additivityofhigherrhoinvariantsandnonrigidityoftopologicalmanifolds, PS16ThesurgeryexactsequenceK-theoryandthesignatureoperator, , JL18additivehigherrhoinvariantforstructuregroupindifferentialpointofview}).

In the literature, there are mainly two  different approaches to the higher $\rho$ invariant associated to homotopy equivalences for the signature operator. The main purpose of the present notes is to give more or less self-contained proof for the equivalence of the two different approaches. 

Let us first briefly recall the two different approaches to the higher $\rho$ invariant. 
\begin{enumerate}
	\item The first approach follows the work of Hilsum and Skandalis \cite{HS92Invarianceparhomotopiedelasignatureacoefficientsdansunfibrepresqueplat}. More precisely,  by following  the work of Hilsum and Skandalis and Wahl  \cite{W13Higherrho-invariantsandthesurgerystructureset}, Piazza and Schick defined a higher $\rho$ invariant  associated to $f$, which lies in the $K$-theory of $D^*(\widetilde{N})^{\Gamma}$ \cite{PS16ThesurgeryexactsequenceK-theoryandthesignatureoperator}.  Here $D^*(\widetilde{N})^{\Gamma}$ is   a geometric $C^*$-algebra whose $K$-theory serves as the receptacle of secondary invariants.
	\item The second approach was given by Higson and Roe. By building upon the work of  Novikov \cite{MR0292913}, Kasparov \cite{MR804752} and Mischenko \cite{MR0268898,MR0362407},  Higson and Roe developed the theory of analytically controlled Hilbert Poincar\'e complexes over $C^\ast$-algebras \cite{HR051MappingsurgerytoanalysisIAnalyticsignatures, HR052MappingsurgerytoanalysisIIGeometricsignatures, HR053MappingsurgerytoanalysisIIIExactsequences}. They defined  a higher $\rho$ invariant  associated to $f$ in this framework, which applies to both the smooth and PL setting. Furthermore,  Higson and Roe constructed a natural transformation (after inverting $2$) from the Browder–Novikov–Sullivan–Wall surgery exact sequence of a closed smooth manifold to a certain exact sequence of $C^\ast$-algebra K-theory groups.  
	This approach was later generalized by Weinberger, Xie and Yu to the Browder–Novikov–Sullivan–Wall surgery exact sequence of a closed topological manifold \cite{WXY18Additivityofhigherrhoinvariantsandnonrigidityoftopologicalmanifolds}. In particular, by combining this approach with a novel description of the topological structure group that was developed in the same paper \cite{WXY18Additivityofhigherrhoinvariantsandnonrigidityoftopologicalmanifolds},  Weinberger, Xie and Yu  proved that the higher $\rho$ invariant defines a group homomorphism from the topological structure group of a given closed topological manifold $N$ to  $K_*(C^*_{L,0}(\widetilde{N})^\Gamma)$, thus answered positively the corresponding long standing open question \cite{WXY18Additivityofhigherrhoinvariantsandnonrigidityoftopologicalmanifolds}. Here $C^*_{L,0}(\widetilde{N})^\Gamma$ is   a geometric $C^*$-algebra whose $K$-theory serves as the receptacle of secondary invariants. In fact, we have $K_{\ast}(C^*_{L,0}(\widetilde{N})^\Gamma) \cong K_{\ast+1}(D^*(\widetilde{N})^{\Gamma})$ (cf. \cite[Section 6]{XY14Positivescalarcurvaturehigherrhoinvariantsandlocalizationalgebras}).  
\end{enumerate}

The approach of Higson and Roe has two different versions: one is bounded and the other is unbounded. Our strategy for proving the equivalence of the above two definitions of the higher $\rho$ invariant is to first show the bound and unbounded versions of the Higson-Roe approach are equivalent, and then show the unbounded version of the Higson-Roe approach is equivalent to the Hilsum-Skandalis approach. 

As usual, the odd dimensional case is completely parallel to the even dimensional case. Alternatively, if we invert $2$, the odd dimensional case reduces to the even dimensional case by taking direct product with the circle. In any case,   without loss of generality, we shall mainly focus on the even dimensional case throughout the article.

The paper is organized as follows. In Section \ref{sec K theory}, we review some geometric $C^*$-algebras and their $K$-theory. In Section \ref{sec Poincare complex}, we review the notion of Hilbert Poincar\'e complexes. In Section \ref{sec K homology}, we introduce constructions of the $K$-homology
class of the signature operator of two particular Hilbert Poincar\'e complexes and the $K$-homology
class of the signature operator on a manifold, and show that they are equivalent to each other. In the last section, we briefly recall different constructions of the higher $\rho$  invariant associated to the homotopy equivalence,  and show that  they are equivalent to each other.

\section{Geometric $C^*$-algebras}\label{sec K theory}

In this section, we recall the definitions of some geometric $C^*$-algebras. Let $N$ be a finite dimensional compact manifold, with fundamental group $\Gamma$.  Then $\Gamma$ acts on $\widetilde{N}$, the universal covering space of $N$, cocompactly and freely. A $\Gamma$-equivariant $\widetilde{N}$-module, denoted as $H_{\widetilde{N}}$, is a separable Hilbert space equipped with a $*$ representation of $C_0(\widetilde{N})$, and a unitary representation of $\Gamma$ on $H_{\widetilde{N}}$, which is covariant to the representation of $C_0(\widetilde{N})$. A $\Gamma$-equivariant $\widetilde{N}$-module is called nondegenerate if the $*$ representation is nondegenerate.  A $\Gamma$-equivariant $\widetilde{N}$-module is called standard if no nonzero function in $C_0(\widetilde{N})$ acts as a compact operator.

\begin{definition}
let $H_{\widetilde{N}}$ be a $\Gamma$-equivariant $\widetilde{N}$-module and $T\in B(H_{\widetilde{N}}) $ be a bounded linear operator.
\begin{itemize}
	\item The propagation of $T$ is defined to be
	\[\sup{\{d(x,y)| (x,y)\in \text{Supp}(T)\}} ,\]
	where $\text{Supp}(T)$ is the complement in $\widetilde{N}\times \widetilde{N}$ of the set of points $(x,y)\in \widetilde{N}\times \widetilde{N}$ for which there exist $f,g\in C_0(\widetilde{N})$ such that $gTf=0$ and $f(x)\neq 0,\ g(y)\neq 0$. $T$ is said to be of finite propagation if its propagation is finite.
	\item $T$ is said to be locally compact if $fT$ and $Tf$ are compact for all $f\in C_0(\widetilde{N})$.
	\item $T$ is said to be pseudo-local if $[T,f]$ is compact for all $f\in C_0(\widetilde{N})$.
\end{itemize}
\end{definition}

\begin{definition}
	Let $H_{\widetilde{N}}$ be a standard nondegenerate $\Gamma$-equivariant $\widetilde{N}$-module.
\begin{itemize}
		\item The $\Gamma$-equivariant Roe algebra $C^*(\widetilde{N})^G$ is the $C^*$-algebra generated by locally compact $\Gamma$-invariant operators with finite propagation.
		\item The so called Roe's structure algebra, $D^*(\widetilde{N})^\Gamma$ is the $C^*$-algebra generated by pseudo-local $\Gamma$-invariant operators with finite propagation. In particular, the $\Gamma$-equivariant Roe algebra $C^*(\widetilde{N})^G$ is an ideal of $D^*(\widetilde{N})^\Gamma$.
		\item The $\Gamma$-equivariant localization algebra $C^*_L(\widetilde{N})^\Gamma$ is the $C^*$-algebra generated by all of the bounded, uniformly norm-continuous functions $f:[0,\infty)\to C^*(\widetilde{N})^G$ such that
		\[ \lim_{t\to \infty}\text{propagation of }f(t)=0. \]

		\item The kernel of the evaluation map
		\begin{eqnarray*}
			\text{ev}: C^*_L(\widetilde{N})^\Gamma &\to & C^*(\widetilde{N})^\Gamma\\
			f   & \to & f(0)
		\end{eqnarray*}
		is defined to be the $\Gamma$-equivariant obstruction algebra $C^*_{L,0}(\widetilde{N})^\Gamma$.
	\end{itemize}
\end{definition}

As shown in \cite{ XY14Positivescalarcurvaturehigherrhoinvariantsandlocalizationalgebras, Y97LocalizationalgebrasandthecoarseBaum-Connesconjecture},  the localized index map
\[\text{Ind}_L: K_1(D^*(\widetilde{N})^\Gamma) \cong K_0(C^*_{L,0}(\widetilde{N})^G)\]
is an isomorphism. Roughly speaking, the $\text{Ind}_L$ map is given by applying index formula to
 paths localizing representatives of classes in $K_1(D^*(\widetilde{N})^\Gamma)$. In this paper, we will use Roe's index formula (\cite{R93CoarsecohomologyandindextheoryoncompleteRiemannianmanifolds}) to define the  $\text{Ind}_L$ map.
 We will not recall the definition of $\text{Ind}_L$ in detail, but only describe the image of some classes in $K_1(D^*(\widetilde{N})^\Gamma)$ under $\text{Ind}_L$ in Subsection \ref{sec def of rho op}.

\section{Hilbert Poincar\'e complexes}\label{sec Poincare complex}
In this section, we recall the definitions of bounded and  unbounded analytically controlled Hilbert Poicnar\'e complexes, the definitions and the homotopy invariance of the signature classes of them.  \cite{HR051MappingsurgerytoanalysisIAnalyticsignatures, HR052MappingsurgerytoanalysisIIGeometricsignatures, HR053MappingsurgerytoanalysisIIIExactsequences}.

\subsection{Analytically controlled Hilbert Poincar\'e complex}

In this subsection, we recall the definition of the analytically controlled Hilbert Poincar\'e complex. We first introduce the definition of the analytically controlled self-adjoint operator.

 \begin{definition}\label{def analytically controlled map bounded}
 A self-adjoint  $\Gamma$-equivariant operator $T$ on $H_{\widetilde{N}}$ (bounded or unbounded) is called  analytically controlled over $\widetilde{N}$ if
 \begin{enumerate}
\item the resolvent operator $(T\pm i)^{-1}$ belongs to $ C^*(\widetilde{N})^\Gamma $,
\item the operator $T(1+T^2)^{-1}$ belongs to $ D^*(\widetilde{N})^\Gamma $.
\end{enumerate}
 \end{definition}

Now we define the analytically controlled complex.
\begin{definition}
An $n$-dimensional $\Gamma$-equivariant
$\widetilde{N}$-complex
\[(H_{\widetilde{N}}, d): H_{0, \widetilde{N}}\stackrel{d}{\rightarrow }  H_{1, \widetilde{N}}\stackrel{d}{\rightarrow } \dots \stackrel{d}{\rightarrow } H_{n-1, \widetilde{N}}\stackrel{d}{\rightarrow } H_{n, \widetilde{N}}\]
consists of $\Gamma$-equivariant  $\widetilde{N}$-modules $H_{*, \widetilde{N}}$ and  $\Gamma$-equivariant differential maps $d$.
It is said to be bounded if all of the differential maps $d$ are bounded operators, unbounded if all of $d$ are unbounded. We say that the complex $(H_{\widetilde{N}}, d)$ is analytically controlled over $\widetilde{N}$ if the operator $D: d+d^*$ is analytically controlled.
\end{definition}

The following is the definition of the analytically controlled  chain homotopy equivalence between analytically controlled Hilbert complexes.

\begin{definition}\label{def ana control homotopy of hil complex bounded}
A chain homotopy equivalence
\[
A: (H_{*, \widetilde{N}}, d) \to (H'_{*,\widetilde{N}}, d')
\]
 between $\Gamma$-equivariant analytically controlled Hilbert complexes over $\widetilde{N}$, is said to be $\Gamma$-equivariant analytically controlled over $\widetilde{N}$ if
\begin{enumerate}
\item all of $A$ belong to $D^*(\widetilde{N})^\Gamma$,
\item there exist chain maps
\[B: (H'_{*, \widetilde{N}}, d') \to (H_{*,\widetilde{N}}, d), B\in D^*(\widetilde{N})^\Gamma, \]
and operators $y, \ y'\in D^*(\widetilde{N})^\Gamma$ with degree $-1$, i.e.
\[
y: H_{i, \widetilde{N}} \to H_{i-1, \widetilde{N}}, \ y': H'_{i, \widetilde{N}} \to H'_{i-1, \widetilde{N}},
\]
such that
    \[
    I-AB= d'y'+y'd', I-BA = dy+yd.
    \]
\end{enumerate}
\end{definition}

The analytically controlled Hilbert Poincar\'e complex is an analytically controlled Hilbert complex equipped with the Hilbert Poincar\'e duality.

\begin{definition}\label{def hilbert Poincare complex bounded}
 A $\Gamma$-equivariant analytically controlled Hilbert Poincar\'e complex is a $\Gamma$-equivariant analytically controlled Hilbert complex
 \[
 (H_{\widetilde{N}}, d): H_{0, \widetilde{N}}\stackrel{d}{\rightarrow }  H_{1, \widetilde{N}}\stackrel{d}{\rightarrow } \dots \stackrel{d}{\rightarrow } H_{n-1, \widetilde{N}}\stackrel{d}{\rightarrow } H_{n, \widetilde{N}},\]
 equipped  with self-adjoint bounded operators $S$ such that
 \begin{enumerate}
 \item $S: H_{i, \widetilde{N}}\to H_{n-i, \widetilde{N}}$,
   \item $Sd^* + dS=0$,
   \item $S$ is a $\Gamma$-equivariant analytically controlled chain homotopy equivalence over $\widetilde{N}$ from the dual complex
       \[(H_{\widetilde{N}}, d^*): H_{n, \widetilde{N}}\stackrel{d^*}{\rightarrow }  H_{n-1, \widetilde{N}}\stackrel{d^*}{\rightarrow } \dots \stackrel{d^*}{\rightarrow } H_{1, \widetilde{N}}\stackrel{d^*}{\rightarrow } H_{0, \widetilde{N}} \]
       to $(H_{\widetilde{N}}, d)$.
 \end{enumerate}
\end{definition}

We mention that to make $S$ to be a genuine chain map,  one need  to employ signs such as
\[
\mathcal{S}(v)=(-1)^p S(v), \forall v\in H_{p,\widetilde{N}}.
\]
However, for the sake of conciseness, we leave it as is, the reader should not be confused. 

Correspondingly, we have the following notion of the analytically controlled homotopy equivalence between Hilbert Poincar\'e complexes.

\begin{definition}\label{def ana control homotopy of hil poinc complex bounded}
Let $(H_{\widetilde{N}}, d, S) $ and $(H'_{\widetilde{N}}, d', S')$ be two analytically controlled Hilbert Poincar\'e complexes over $\widetilde{N}$.
Let
\[
A: (H_{\widetilde{N}}, d) \to (H'_{\widetilde{N}}, d')
\]
be a  $\Gamma$-equivariant analytically controlled chain homotopy equivalence. Then the homotopy equivalence  $A$ is said to be a $\Gamma$-equivariant analytically controlled chain homotopy equivalence between  $(H_{\widetilde{N}}, d, S) $ and $(H'_{\widetilde{N}}, d', S')$, if there exist operators $y: H_{p,\widetilde{N}} \to H_{n-1-p\widetilde{N}}$, $y\in D^*(\widetilde{N})^\Gamma$, such that
\[
(-i)^{p(p-1)+l}(AS'A^*-S)(v)= dy  + yd^*(v), \forall v\in  H_{p,\widetilde{N}},
\]
where $l=\frac{n}{2}$ if $n$ is even, and $l=\frac{n-1}{2}$ if $n$ is odd. 
\end{definition}

\subsection{The signature class}

In this subsection, we define the signature class associated to a Hilbert Poincar\'e complex.

For both of the bounded and the unbounded $\Gamma$-equivariant Hilbert Poincar\'e complexes analytically controlled over $\widetilde N$, their signature classes live in $K_*(C^*(\widetilde{N})^G )$.

Let $(H_{\widetilde{N}}, d, S)$ be a $\Gamma$-equivariant analytically controlled Hilbert Poincar\'e complex (bounded or unbounded) over $\widetilde{N}$. Let $D$ be $d+d^*.$  In \cite{HR051MappingsurgerytoanalysisIAnalyticsignatures}, it was shown that operators
\[D+ S,\ D-S\]
are both invertible.

When $(H_{\widetilde{N}}, d, S)$ is odd dimensional, by definition, one can see that the invertible operator
\[
(D+S)(D-S)^{-1}=I+2S(D-S)^{-1}: \oplus_{k} H_{2k, \widetilde{N}}\to \oplus_{k} H_{2k, \widetilde{N}}
\]
belongs to $C^*(\widetilde{N})^{\Gamma,+}$, hence defines a class in $K_1(C^*(\widetilde{N})^\Gamma )$.

Higson and Roe proved the following lemma (\cite[Lemma 5.7]{HR051MappingsurgerytoanalysisIAnalyticsignatures}).
\begin{lemma}\label{lemma odd functional cal for perturb}
Let $i=\sqrt{-1}$, the operators $i\pm (D\pm S)$ are invertible, and $(i\pm (D\pm S))^{-1}$ belong to $C^*(\widetilde{N})^\Gamma$.
\end{lemma}

By Lemma \ref{lemma odd functional cal for perturb}, one can see that for any  bounded function $g:(-\infty, \infty)\to \mathbb{R}$, all of the operators $g(D\pm S)$ and $g(D)$ belong to $D^*(\widetilde{N})^{\Gamma, +}$.

On the other hand, Higson and Roe also proved the following lemma (\cite[Lemma 5.8]{HR051MappingsurgerytoanalysisIAnalyticsignatures})

\begin{lemma}\label{lemma even functional cal for perturb}
For any bounded function 
\[g:(-\infty, \infty)\to \mathbb{R},\]
 we have $g(D\pm S)-g(D)$ belong to $C^*(\widetilde{N})^\Gamma$.
\end{lemma}

Lemma \ref{lemma even functional cal for perturb} implies that
\[
P_+(D+S)-P_+(D-S)
\]
belongs to $C^*(\widetilde{N})^{\Gamma,+}$, where $P_+(D\pm S)$ are the positive projections of the invertible self-adjoint operators $D\pm S$ respectively. Since  $P_+(D\pm S)$ are projections in $D^*(\widetilde{N})^{\Gamma, +}$ whose difference belongs to $C^*(\widetilde{N})^{\Gamma,}$,  they give rise to a class in $K_0(C^*(\widetilde{N})^{\Gamma})$ represented by the following formal difference
\[
[P_+(D+S)]-[P_+(D-S)]\in K_0(C^*(\widetilde{N})^{\Gamma}).
\]

\begin{definition}
\begin{enumerate}

\item Let $(H_{\widetilde{N}}, d, S)$ be an odd-dimensional $\Gamma$-equi-\\variant analytically controlled Hilbert Poincar\'e complex over $\widetilde{N}$. Its signature is the class in $K_1( C^*(\widetilde{N})^\Gamma )$ of the invertible operator
\[
(D+S)(D-S)^{-1} : \oplus_k H_{2k, \widetilde{N}} \to  \oplus_k H_{2k, \widetilde{N}}.
\]

\item Let $(H_{\widetilde{N}}, d, S)$ be an even-dimensional $\Gamma$-equivariant analytically controlled Hilbert Poincar\'e complex over $\widetilde{N}$. Its signature is the class in $K_0( C^*(\widetilde{N})^\Gamma )$ determined by the following formal difference
\[
[P_+(D+S)]-[P_+(D-S)].
\]

\end{enumerate}
\end{definition}

In the following, we consider only even-dimensional complexes.

\subsection{Homotopy invariance of the Signature class}\label{subsec homo inva of the signature class}

In this subsection, we shall briefly recall Higson and Roe's proof of the invariance of the signature class of the $\Gamma$-equivariant analytically controlled Hilbert Poincar\'e complex over $\widetilde{N}$ under the $\Gamma$-equivariant analytically controlled chain homotopy equivalence (\cite{HR051MappingsurgerytoanalysisIAnalyticsignatures}).
This proof is important to the construction of the higher $\rho$ invariant associated to a homotopy equivalence.

Let
\[
A: (H_{\widetilde{N}}, d, S)\to (H'_{\widetilde{N}}, d', S')
\]
 be a homotopy equivlence between two $\Gamma$-equivariant analytically controlled Hilbert Poincar\'e complexes over $\widetilde{N}$. We set
\[\overline{d} = \left(
\begin{array}{cc}
d & 0 \\
0   & d'
\end{array}
\right)\ \text{and}\ \overline{D}= \overline{d}+ \overline{d}^* .\]
Let $\overline{S}$ be
\[
 \left(
\begin{array}{cc}
S & 0 \\
0 & -S'
\end{array}
\right).
\]
Consider the $\Gamma$-equivariant analytically controlled Hilbert Poincar\'e complex over $\widetilde{N}$
\[
(H_{\widetilde{N}} \oplus H'_{\widetilde{N}},
\overline{d} ,\overline{S} ).
\]
Then its signature class in $K_0(C^*(\widetilde{N})^\Gamma)$ is determined by the following formal difference
\[
[P_+(\overline{D} + \overline{S})]- [P_+(\overline{D}-\overline{S})].
\]
By definition,
\[
\left(
\begin{array}{cc}
 & A^*\\
1  &
\end{array}
\right)
\left(
\begin{array}{cc}
-S' & 0\\
0  & S'
\end{array}
\right)
\left(
\begin{array}{cc}
 & 1\\
A &
\end{array}
\right)=\begin{pmatrix}
 A^* S' A  &  \\
 & -S'
\end{pmatrix}
\]
induces the same homology isomorphism with
\[
\left(
\begin{array}{cc}
S & 0\\
0  & -S'
\end{array}
\right).
\]
Thus for any $t\in [0,1]$,
\[
(H_{\widetilde{N}} \oplus H'_{\widetilde{N}},
\overline{d} ,  \left(
\begin{array}{cc}
(1-t)S + t A^* S' A & 0\\
0 &  -S'
\end{array}
\right)),
\]
are all $\Gamma$-equivariant analytically controlled Hilbert Poincar\'e complex over $\widetilde{N}$. Direct computation  shows that for any $t\in [0,1]$,
\[
( H_{\widetilde{N}} \oplus H'_{\widetilde{N}},
\overline{d},
\left(
\begin{array}{cc}
\cos (\frac{\pi}{2}t) A^* S' A & \sin (\frac{\pi}{2}t) A^* S'  \\
\sin (\frac{\pi}{2}t)   S' A   & - \cos (\frac{\pi}{2}t) S'
\end{array}
\right)),
\]
and
\[
( H_{\widetilde{N}} \oplus H'_{\widetilde{N}},
\overline{d},
 \left(
\begin{array}{cc}
	0 &  e^{i\pi t}S' A\\
	e^{-i\pi t}A^*S'   & 0
\end{array}
\right) )
\]
are all $\Gamma$-equivariant analytically controlled Hilbert Poincar\'e complex over $\widetilde{N}$.

Set
\[
P[A](t)=P_+\left(	\overline{D}+
\begin{pmatrix}
(1-t)S + t A^* S' A & 0\\
0 &  -S'
\end{pmatrix}
\right)
\]
when $t\in [0,1]$,
\[P[A](t)=P_+\left(	\overline{D}+   \begin{pmatrix}
\cos (\frac{\pi}{2}(t-1)) A^* S' A & \sin (\frac{\pi}{2}(t-1)) A^* S'  \\
\sin (\frac{\pi}{2}(t-1))   S' A   & - \cos (\frac{\pi}{2}(t-1)) S'
\end{pmatrix}\right)\]
when $t\in [1,2]$ and
\[
P[A](t)=P_+\left(	\overline{D}+
\begin{pmatrix}
	0 & S' A\\
	A^*S'   & 0
\end{pmatrix}
\right)\]  	
when $t\in [2,3]$. Then the path $P[A]: [0,3]\to D^*(\widetilde{N})^\Gamma$ is a continuous path of projections in $D^*(\widetilde{N})^\Gamma$.

Similarly, set
\[Q[A](t)=P_+\left(	\overline{D}- \begin{pmatrix}
(1-t)S + t A^* S' A & 0\\
0 &  -S'
\end{pmatrix}\right)\]
when $t\in [0,1]$,
\[ Q[A](t)=P_+(	\overline{D}- \left(
\begin{pmatrix}
\cos (\frac{\pi}{2}(t-1)) A^* S' A & \sin (\frac{\pi}{2}(t-1)) A^* S'  \\
\sin (\frac{\pi}{2}(t-1))   S' A   & - \cos (\frac{\pi}{2}(t-1)) S'
\end{pmatrix}\right)
\]
when $t\in [1,2]$
and
\[Q[A](t)=P_+\left(	\overline{D}-
\begin{pmatrix}
	0 &  e^{i\pi (t-2)}S' A\\
	e^{-i\pi (t-2)}A^*S'   & 0
\end{pmatrix}
\right)\]
when $t\in [2,3]$.  Then the path 
\[Q[A]: [0,3]\to D^*(\widetilde{N})^\Gamma\]
 is a continuous path of projections in $D^*(\widetilde{N})^\Gamma$.

By Lemma \ref{lemma odd functional cal for perturb} and Lemma \ref{lemma even functional cal for perturb}, one can see that
\[
P[A](t)-Q[A](t)\in C^*(\widetilde{N})^\Gamma, t\in [0,3],
\]
and that
\[
P[A](3)=Q[A](3).
\]
Hence the signature class of the  complex
\[
(H_{\widetilde{N}} \oplus H'_{\widetilde{N}},
\overline{d} ,\overline{S} ).
\]
is trivial, i.e. the signature class of the complex
\[
(H_{\widetilde{N}}, d, S)
\]
equals the one of the complex
\[
(H'_{\widetilde{N}}, d', S').
\]

\section{$K$-homology class of the signature operator}\label{sec K homology}
In this section, we introduce two   $\Gamma$-equivariant analytically controlled Hilbert Poincar\'e complexes over $\widetilde{N}$. We show that the $K$-homology class associated to  each of two complexes coincides with the $K$-homology class of the signature operator on $N$. 

\subsection{$L^2$ simplicial cochain complex}\label{subsec L2 simplicial and bounded complex}

In this subsection, we recall the definition of a bounded Hilbert Poincar\'e complex, and recall the definition of  the $K$-homology class of the signature operator on a manifold associated to this complex.

For an $n=2k$ dimensional manifold $N$ with fundamental group $\Gamma$, equip $\widetilde{N}$ with a $\Gamma$-invariant triangulation $\text{Tri}(\widetilde{N} )$.  The  $L^2$-completion of the  simplicial cochain complex $(E^*(\widetilde{N}) , b_{\widetilde N})$ given by the triangulation then induces a bounded $\Gamma$-equivariant analytically controlled Hilbert complex over $\widetilde{N}$,
\[(L^2(E^*(\widetilde{N})),
b_{\widetilde N}).
\]
The Poincar\'e duality map $T_{\widetilde N}$ of $( L^2(E^*(\widetilde{N})), b_{\widetilde N})$ is given by taking the cap product with the fundamental class $[\widetilde N]$. Define operators $S_{\widetilde N}$ by
\[
S_{\widetilde N} (v) = i^{p(p-1)+ \frac{n}{2}} T_{\widetilde N}(v),\ \forall v\in L^2(E^p(\widetilde{N})).
\]
Then $S_{\widetilde N}$ satisfies
\[
 S_{\widetilde N}^*=S_{\widetilde N} , S_{\widetilde N} b_{\widetilde N}= - b_{\widetilde N}^* S_{\widetilde N}, S^2_{\widetilde N}=1.
\]
By definition,
\[(L^2(E^*(\widetilde{N})),
b_{\widetilde N}, S_{\widetilde N})\]
is an analytically controlled Hilbert Poincar\'e complex. Recall that its signature class is defined as
\[
[P_+(B_{\widetilde N}+S_{\widetilde N})]-[P_+(B_{\widetilde N}-S_{\widetilde N})]\in K_0(C^*(\widetilde N)^\Gamma),
\]
where $B_{\widetilde N}$ is $b_{\widetilde N}+b^*_{\widetilde N}.$
We will denote the signature class of the complex
\[(L^2(E^*(\widetilde{N})),
b_{\widetilde N}, S_{\widetilde N})\]
by
\[
\text{Signature}_{PL}(\widetilde N).
\]

In the following, we briefly recall a  procedure for defining the $K$-homology class of the signature operator in $K_0(C^*_L(\widetilde{N})^\Gamma)$ by the simplicial cochain complex. This procedure was first given in  \cite[Appendix B]{WXY18Additivityofhigherrhoinvariantsandnonrigidityoftopologicalmanifolds} and \cite{HX18HighersignaturesofWittspaces}.

Consider the $\Gamma$-invariant triangulation $\text{Tri}(\widetilde{N} )$.  There exits a refinement procedure for the triangulation $\text{Tri}(\widetilde{N} )$ so that all of the successive refinements have bounded geometry uniformly,  cf. \cite[Section 4.2]{WXY18Additivityofhigherrhoinvariantsandnonrigidityoftopologicalmanifolds} and \cite{HX18HighersignaturesofWittspaces}. More precisely, this procedure generates a series of $\Gamma$-equivariant  subdivisions
 \[
\text{Sub}^0(\widetilde{N} )=\text{Tri}(\widetilde{N} ), \ \text{Sub}^n(\widetilde{N} ): =\text{Sub}( \text{Sub}^{n-1}(\widetilde{N} ) ), n\in \mathbb{N}_+,
 \]
for which there is a positive integer $k$, such that each of the vertices in $\text{Sub}^n(\widetilde{N} )$ lies in at most $k$ different simplices for all $n$.

Recall that every locally finite simplicial complex carries a natural metric, whose restriction to each $n$-simplex is the Riemannian metric obtained by identifying this $n$-simplex with the standard one in the Euclidean space $\mathbb{R}^n$. Such a metric is called the simplicial metric. Let $\widetilde{N}_n, n \in \mathbb{N}$ be the manifold $\widetilde{N}$  equipped with the $\Gamma$-invavariant  triangulation $\text{Sub}^n(\widetilde{N})$ and the corresponding simplicial metric. Set
\[
b_{\widetilde N, \mathbb{N}}= \oplus_n b_{\widetilde {N}_n}, S_{\widetilde N, \mathbb{N}} = \oplus_n S_{\widetilde{N}_n}.
\]
The signature class
\[
[P_+(B_{\widetilde N, \mathbb{N}}+S_{\widetilde N, \mathbb{N}})]-[ P_+(B_{\widetilde N, \mathbb{N}}-S_{\widetilde N, \mathbb{N}})]
\]
of the complex
\[
(L^2(E^*(\sqcup_n \widetilde{N}_n)),  b_{\widetilde N, \mathbb{N}},S_{\widetilde N, \mathbb{N}})
\]
gives an element in $K_0(C^*( \sqcup_n \widetilde{N}_n )^\Gamma)$. Due to the construction in Subsection \ref{subsec homo inva of the signature class}, the controlled homotopy equivalence
\[
A_n : (L^2(E^*( \widetilde{N}_n)),  b_{\widetilde{N}_n},S_{\widetilde{N}_n})\to ( L^2(E^*(\widetilde{N}_{n+1})),  b_{\widetilde{N}_{n+1}},S_{\widetilde{N}_{n+1}})
\]
produces a path of  differences %of projections
\[P_{\widetilde N, n}[A_n](t)-Q_{\widetilde N,n}[A_n](t), t\in [0,1]\]
connecting
\[
P_+(B_{\widetilde{N}_n}+S_{\widetilde{N}_n})- P_+(B_{\widetilde{N}_n}-S_{\widetilde{N}_n})\]
and
 \[ P_+(B_{\widetilde{N}_{n+1}}+S_{\widetilde{N}_{n+1}})- P_+(B_{\widetilde{N}_{n+1}}-S_{\widetilde{N}_{n+1}}).
\]
For $t\in [0, \infty), $ define $F_{PL,\widetilde N}(t)$ as
\[
 \left\{\begin{array}{cc}
P_+(B_{\widetilde{N}_t}+S_{\widetilde{N}_t})- P_+(B_{\widetilde{N}_t}-S_{\widetilde{N}_t}) & t\in \mathbb{N},\\
P_{\widetilde N, n}[A_n](t-n)-Q_{\widetilde N,n}[A_n](t-n) & n\leq t\leq n+1, n\in \mathbb{N}.
\end{array}
\right.
\]
Scale the simplicial metric of $\widetilde{N}_n$ back to the original metric of $\widetilde{N}$, then $F_{PL,\widetilde N}$ represents a class in $K_0(C^*_L( \widetilde{N})^\Gamma)$. In fact, by construction, one can see that
\[
\lim_{t\to \infty} \text{propagation}\ F_{PL,\widetilde N}(t)=0,
\]
and
\[
[F_{PL, \widetilde N}(0)]= \text{Signature}_{PL}(\widetilde N).
\]
In the following, we refer to  $[F_{PL, \widetilde N}]$ as the $K$-homology class of the signature operator associated to the complex
\[(L^2(E^*(\widetilde{N})),
b_{\widetilde N}, S_{\widetilde N}).\]

\subsection{$L^2$ deRham complex}\label{subsec L 2 form and unbounded com}

In this subsection, we recall the definition of an unbounded Hilbert Poincar\'e complex, and give the definition of  the $K$-homology class of the signature operator on a manifold associated to this complex.

Recall that $N$ is an $n=2k$ dimensional compact manifold.
Let $L^2(\wedge^p (\widetilde{N}))$ be the Hilbert space of $L^2$ $p-$forms. We further set $d_{\widetilde N}$ to be the $\Gamma$-invariant differentials and $\tau_{\widetilde N}$ to be the chiral  dualities, i.e
\[
\tau_{\widetilde N}(f)=i^{p(p-1)+ \frac{n}{2}} \star f,\ \forall f\in L^2(\wedge^p (\widetilde{N})),
\]
where $\star$ is the Hodge $\star$ operator.
 Then we have an unbounded analytically controlled Hilbert Poincar\'e complex:
\[
(L^2(\wedge^* (\widetilde{N})), d_{\widetilde N}, \tau_{\widetilde N}).
\]
The signature class of this complex is then determined by the formal difference
\[
\text{Signature}_{C^\infty}(\widetilde N):=[P_+(D_{\widetilde N} + \tau_{\widetilde N})]- [P_+(D_{\widetilde N}-\tau_{\widetilde N})].
\]

In \cite[Theorem 5.12]{HR052MappingsurgerytoanalysisIIGeometricsignatures}, it was proved that for compact manifold $N$, the usual  homotopy equivalence
\[H: (L^2(\wedge^* (\widetilde{N})), d_{\widetilde N})  \to  (L^2(E^*(\widetilde{N})), b_{\widetilde N}) \]
between the deRham complex and the simplicial cochain complex is an analytically controlled homotopy equivalence of analytically controlled Hilbert Poincar\'e complexes:
\[H: (L^2(\wedge^* (\widetilde{N})), d_{\widetilde N}, \tau_{\widetilde N})  \to  (L^2(E^*(\widetilde{N})), b_{\widetilde N}, S_{\widetilde N}). \]
Thus we have
\[\text{Signature}_{C^\infty}(\widetilde N)=\text{Signature}_{PL}(\widetilde N).\]

We now construct a representative, $F_{C^\infty,\widetilde N}$, of the $K$-homology class of the signature operator associated to the complex
\[
(L^2(\wedge^* (\widetilde{N})), d_{\widetilde N}, \tau_{\widetilde N}).
\]

For any $t\in [0,1]$, we consider $\sqcup_{n\in \mathbb{N}} \widetilde{N}_{n+t}$, where $\widetilde{N}_{n+t}$ is the manifold $\widetilde{N}$ equipped  with the metric $(n+1+t)$ times of the original one. Consider
\[
d_{\widetilde{N}, \mathbb{N} + t}=
	d_{\sqcup_{n\in \mathbb{N}} \widetilde{N}_{n+t}},
\tau_{\widetilde{N}, \mathbb{N} + t}=
	\tau_{\sqcup_{n\in \mathbb{N}} \widetilde{N}_{n+t}}.
\]
Then the signature class of
\[(L^2(\wedge^* (\sqcup_{n\in \mathbb{N}} \widetilde{N}_{n+t} ) ) , d_{\widetilde{N}, \mathbb{N} + t}, \tau_{\widetilde{N}, \mathbb{N} + t})\]
in $K_0(C^*(\sqcup_{n\in \mathbb{N} } \widetilde{N}_{n+t})^\Gamma)$ then equals the formal difference
\[
[P_+( D_{\widetilde{N}, \mathbb{N} + t}+\tau_{\widetilde{N}, \mathbb{N} + t})]-[P_+( D_{\widetilde{N}, \mathbb{N} + t}-\tau_{\widetilde{N}, \mathbb{N} + t})].
\]

Scale $\widetilde{N}_{n+t}$ back to $\widetilde{N}$ and range over all of $t$ in $[0,1]$, the path of differences of projections
\[
n+t\mapsto P_+( D_{\widetilde{N}, n + t}+\tau_{\widetilde{N}, n+ t})-P_+( D_{\widetilde{N}, n+ t}-\tau_{\widetilde{N}, n + t})
\]
becomes an element in $C^*_L(\widetilde{N})^\Gamma$. We denote this element as $F_{C^\infty,\widetilde{N}},$ and call the class of the corresponding formal difference $[F_{C^\infty,\widetilde{N}}]\in K_0(C^*_L(\widetilde{N})^\Gamma)$ as the K-homology class of the signature operator  associated to the complex
\[(L^2(\wedge^* (\widetilde{N})), d_{\widetilde N}, \tau_{\widetilde N}).\]
By construction, one can see that
\[
[F_{C^\infty,\widetilde N}(0)]=\text{Signature}_{C^\infty}(\widetilde N).
\]

There is an equivalent way to define $F_{C^\infty,\widetilde N}$, which is similar to the definition of $F_{PL,\widetilde N}$. In fact, the identity map
 \begin{eqnarray*}
 I_n : \widetilde{N}_n&\to& \widetilde{N}_{n+1}\\
  x&\mapsto &x
 \end{eqnarray*}
 is a controlled homotopy equivalence for any $n\in \mathbb{N}$.  Set
\[
d_{\widetilde{N}, \mathbb{N}}= \oplus_n d_{\widetilde{N}_n}, \tau_{\widetilde{N}, \mathbb{N}} = \oplus_n \tau_{\widetilde{N}_n}.
\]
The signature class
\[
[P_+(D_{\widetilde{N}, \mathbb{N}}+\tau_{\widetilde{N}, \mathbb{N}})]- [P_+(D_{\widetilde{N}, \mathbb{N}}-\tau_{\widetilde{N}, \mathbb{N}})]
\]
of
\[
(L^2(\wedge^*(\sqcup_n \widetilde{N}_n)),  d_{\widetilde{N}, \mathbb{N}},\tau_{\widetilde{N}, \mathbb{N}})
\]
gives an element in $K_0(C^*( \sqcup_n \widetilde{N}_n )^\Gamma)$. Then by the construction in Subsection \ref{subsec homo inva of the signature class}, the controlled homotopy equivalence
\[
I_n: (L^2(\wedge^* ( \widetilde{N}_n)),  d_{\widetilde{N}_n},\tau_{\widetilde{N}_n})\to ( L^2(\wedge^* ( \widetilde{N}_{n+1})),  d_{\widetilde{N}_{n+1}},\tau_{\widetilde{N}_{n+1}})
\]
produces a path of difference of projections
\[t\mapsto P_{\widetilde{N}, n}[I_n](t)-Q_{\widetilde{N}, n}[I_n](t), t\in [0,1]\]
 connecting
\[
P_+(D_{\widetilde{N}_n}+\tau_{\widetilde{N}_n})- P_+(D_{\widetilde{N}_n}-\tau_{\widetilde{N}_n})\]
and
 \[
 P_+(D_{\widetilde{N}_{n+1}}+\tau_{\widetilde{N}_{n+1}})- P_+(D_{\widetilde{N}_{n+1}}-\tau_{\widetilde{N}_{n+1}}).
\]
For $t\in [0, \infty)$, set $F_{C^\infty,\widetilde{N}}'(t)$
as
\[P_+(D_{\widetilde{N}_t}+\tau_{\widetilde{N}_t})- P_+(D_{\widetilde{N}_t}-\tau_{\widetilde{N}_t})\]
when $t\in \mathbb{N}_+,$ and as
\[P_{\widetilde{N}, n}[I_n](t-n)-Q_{\widetilde{N}, n}[I_n](t-n)\]
when $ n\leq t\leq n+1, n\in \mathbb{N}.$

Scale the metric of $\widetilde{N}_n$ back to the original metric of $\widetilde{N}$, $F_{C^\infty,\widetilde{N}}'$ becomes an element in $C^*_L( \widetilde{N} )^\Gamma$. Apply the argument in Subsection \ref{subsec homo inva of the signature class} to any $t\in [0,+\infty), $ it is not difficult to see that
$[F_{C^\infty,\widetilde{N}}']=[F_{C^\infty,\widetilde{N}}]\in K_0(C^*_L( \widetilde{N} )^\Gamma).$

The following lemma follows from constructions of $F_{C^\infty, \widetilde{N}}$ and $F_{PL, \widetilde{N}}$.
\begin{lemma}\label{lemma equiv of F PL and smooth}
We have $[F_{C^\infty, \widetilde{N}}]= [F_{PL, \widetilde{N}}]\in  K_0(C^*_L( \widetilde{N} )^\Gamma).$
\end{lemma}
\begin{proof}
 Let $\sqcup_n H_n$ be the homotopy equivalence of controlled complexes
\begin{eqnarray*}
\sqcup_n H_n&:&(L^2(\wedge^* (\sqcup_{n\in \mathbb{N}} \widetilde{N}_{n} ) )  , d_{\widetilde{N}, \mathbb{N} }, \tau_{\widetilde{N},  \mathbb{N}})\\
&&\to (L^2(\sqcup_{n\in \mathbb{N}}  E^*(\widetilde{N}_n)), b_{\widetilde{N}, \mathbb{N}},  S_{\widetilde{N}, \mathbb{N}}) .
\end{eqnarray*}
 Thus we can connect
\[P_+(B_{\widetilde N, \mathbb{N}}+S_{\widetilde N, \mathbb{N}})-P_+(B_{\widetilde N, \mathbb{N}}-S_{\widetilde N, \mathbb{N}})\]
to
\[
P_+( D_{\widetilde N, \mathbb{N}}+\tau_{\widetilde N, \mathbb{N}})-P_+( D_{\widetilde N, \mathbb{N}}-\tau_{\widetilde N, \mathbb{N}})
\]
by a path of  differences of projections
\[
P_{\widetilde N, \mathbb{N} }[\sqcup_n H_n](s)- Q_{ \widetilde N,\mathbb{N} }[\sqcup_n H_n](s), s\in [0,1],
\]
 which determines $K_0(C^*(\sqcup_n \widetilde{N}_n )^\Gamma)$ class for any $s$.
Then the homotopy equivalence $A_n\oplus I_n $ from
\[
(E^*(\widetilde{N}_n), b_{\widetilde N_n},  S_{\widetilde N_n}) \oplus (L^2(\wedge^* ( \widetilde{N}_n ) )  , d_{\widetilde N_n}, \tau_{\widetilde N_n}) \]
to
\[(E^*(\widetilde{N}_{n+1}), b_{\widetilde N_{n+1}},  S_{\widetilde N_{n+1}}) \oplus (L^2(\wedge^* ( \widetilde{N}_{n+1} ) )  , d_{\widetilde N_{n+1} }, \tau_{\widetilde N_{n+1}})
\]
produces pathes connecting the element $P_{\widetilde N, n }[\sqcup_n H_n](s)- Q_{ \widetilde N,n }[\sqcup_n H_n](s)$ to the element $ P_{\widetilde N, n+1 }[\sqcup_n H_n](s)- Q_{\widetilde N, n+1 }[\sqcup_n H_n](s)$, $\forall s\in [0,1].$
Scale $\widetilde N_n$ back to $\widetilde N$, one can obtain a path of differences of projections in $C^*_L(\widetilde N)^\Gamma$ that connects $F_{PL,\widetilde N}$ to $F'_{C^\infty,\widetilde N}$, which finishes the proof.
\end{proof}

\subsection{Index theoretic approach}\label{sub index theo app of k homo}

In this subsection, we recall the index theoretic approach to the definition of the $K$-homology class of the signature operator on a manifold.

On an even dimensional manifold $N$ with fundamental group $\Gamma$, let $D_{\widetilde N}$ be the signature operator on $\widetilde{N}$, and  $\tau_{\widetilde N}$ be the  chiral duality operator.  Recall that $\tau_{\widetilde N}$ is a symmetry, that is $\tau_{\widetilde N}^*=\tau_{\widetilde N}, \ \tau_{\widetilde N}^2=1$.

A normalizing function is a continuous odd function  $\phi:\mathbb{R} \to [-1,1]$ such that $\lim\limits_{x\to +\infty}\phi(x)=1$. Set $g(x)= \frac{x}{\sqrt{1+x^2}}, \ h(x)= \sqrt{1-g^2(x)}$.

Define operators $G_{\widetilde{N}}(t),t\in [0,+\infty)$ as
\[
	g(\frac{D_{\widetilde{N}}}{t+1} ) - \tau_{\widetilde{N}} h(\frac{D_{\widetilde{N}}}{t+1}).
\]
Since $g$ can be approximated by normalizing functions $\{\phi_n\}$ whose Fourier transform have compact support, and that the propagation of
\[
\phi(\frac{D_{\widetilde{N}}}{t+1})=\int_{-\infty}^{\infty} \hat{\phi}(s)e^{2\pi i s \frac{D_{\widetilde{N}}}{t+1}}d s
\]
tends to $0$ as $t$ tends to $\infty,$ the path $t\mapsto G_{\widetilde{N}}(t), t\in [0,+\infty),$ defines a symmetry in $D^*_L(\widetilde{N})^{\Gamma}$.
By definition, we have
\[t\mapsto  G_{\widetilde{N}}(t)\tau_{\widetilde{N}}+\tau_{\widetilde{N}} G_{\widetilde{N}}(t) \in C^*_L(\widetilde{N})^\Gamma.\]

 Define
 \[
 F_{\widetilde{N}}:=P_+(-G_{\widetilde{N}}\tau_{\widetilde{N}} G_{\widetilde{N}})-P_+(\tau_{\widetilde{N}})\in C^*_L(\widetilde{N})^{\Gamma}.
 \]
 In fact, $F_{\widetilde{N}}(t)$ is defined by applying Roe's index formula in \cite{R93CoarsecohomologyandindextheoryoncompleteRiemannianmanifolds} to $g(\frac{D_{\widetilde{N}}}{t+1} )$.

Then the following formal difference
\[
[F_{\widetilde{N}}]:=[P_+(-G_{\widetilde{N}}\tau_{\widetilde{N}} G_{\widetilde{N}} ) ]-[P_+(\tau_{\widetilde{N}})]
\]
forms an element in $K_0(C_L^*(\widetilde{N})^G).$  We call this element the $K$-homology class of the signature operator $D_{\widetilde{N}},$ or simply the  $K$-homology class of the signature operator on $N.$

In the following, we shall give a different interpretation of $F_{\widetilde{N}}.$ For any $t\in [0,+\infty),$
\[(L^2(\wedge^* (\widetilde{N})), \frac{d_{\widetilde N}}{t+1}, \tau_{\widetilde N})\]
is an analytically controlled Hilbert Poincar\'e complex, thus
\[
\frac{D_{\widetilde{N}}}{t+1} \pm \tau_{\widetilde N}
\]
are invertible.

\begin{lemma}\label{lemma equivalent def of F op}
The path
\[
F_{\widetilde{N}}':=t\mapsto P_+(\frac{D_{\widetilde{N}}}{t+1}+\tau_{\widetilde N})-P_+(\frac{D_{\widetilde{N}}}{t+1}-\tau_{\widetilde N})
\]
belongs to $C^*_L(\widetilde{N})^\Gamma$, thus the following formal difference
\[
[F_{\widetilde{N}}']= [t\mapsto P_+(\frac{D_{\widetilde{N}}}{t+1}+\tau_{\widetilde N})]-[t\mapsto P_+(\frac{D_{\widetilde{N}}}{t+1}-\tau_{\widetilde N})]
\]
lies in $K_0(C^*_L(\widetilde{N})^\Gamma)$. We have
\[
[F_{\widetilde{N}}']=[F_{\widetilde{N}}]\in K_0(C^*_L(\widetilde{N})^\Gamma).
\]
\end{lemma}
\begin{proof}

 	Note that
 	\[
 	( G_{\widetilde{N}}(t)- \tau_{\widetilde{N}} )^2= 2( 1+ f(\frac{D_{\widetilde{N}}}{t+1}) )
 	\]
 	are invertible with uniform lower bound $2$ for $t\in [0,+\infty)$. Thus
 	\[
 	I-\tau_{\widetilde{N}}G_{\widetilde{N}}=\tau_{\widetilde{N}}( \tau_{\widetilde{N}}-G_{\widetilde{N}})
 	\]
 	is invertible in $C^*_L(\widetilde{N})^{\Gamma,+}$, which implies that $1$ does not belongs to the spectrum of $\tau_{\widetilde{N}}G_{\widetilde{N}}$.

	Direct computation shows that
 	\begin{eqnarray*}
 		 \tau_{\widetilde{N}}\sqrt{\tau_{\widetilde{N}}G_{\widetilde{N}}(t)} &=& \sqrt{\tau_{\widetilde{N}}G_{\widetilde{N}}(t)}^*\tau_{\widetilde{N}},\\
 		G_{\widetilde{N}}(t)\sqrt{\tau_{\widetilde{N}}G_{\widetilde{N}}(t)} &=& \sqrt{\tau_{\widetilde{N}}G_{\widetilde{N}}(t)}^*G_{\widetilde{N}}(t) .
 	\end{eqnarray*}
 	Let $W_{\widetilde{N}}(t)$ be
 	\[
 	\sqrt{\tau_{\widetilde{N}}G_{\widetilde{N}}(t)}G_{\widetilde{N}}(t).
 	\]
    Thus we have
 	\[
 		W_{\widetilde{N}}(t)^* (- G_{\widetilde{N}}(t) \tau_{\widetilde{N}} G_{\widetilde{N}}(t))W_{\widetilde{N}}(t)
 		= -\tau_{\widetilde{N}}G_{\widetilde{N}}(t)\tau_{\widetilde{N}} 
 		= \frac{\frac{D_{\widetilde{N}}}{t+1}+ \tau_{\widetilde{N}}}{|\frac{D_{\widetilde{N}}}{t+1}+ \tau_{\widetilde{N}}|},
 	\]
%\begin{eqnarray*}
% 		&&W_{\widetilde{N}}(t)^* (- G_{\widetilde{N}}(t) \tau_{\widetilde{N}} G_{\widetilde{N}}(t))W_{\widetilde{N}}(t)\\
% 		&=& -\tau_{\widetilde{N}}G_{\widetilde{N}}(t)\tau_{\widetilde{N}} \\
% 		&=& \frac{\frac{D_{\widetilde{N}}}{t+1}+ \tau_{\widetilde{N}}}{|\frac{D_{\widetilde{N}}}{t+1}+ \tau_{\widetilde{N}}|}.
% 	\end{eqnarray*}
 and
 \[
 W_{\widetilde{N}}(t)^*(\tau_{\widetilde{N}}) W_{\widetilde{N}}(t) =G_{\widetilde{N}}(t)=\frac{\frac{D_{\widetilde{N}}}{t+1}- \tau_{\widetilde{N}}}{|\frac{D_{\widetilde{N}}}{t+1}- \tau_{\widetilde{N}}|}.
 \]
 	Hence we have
 	\begin{eqnarray*}
  &&W_{\widetilde{N}}(t)^*(P_+(- G_{\widetilde{N}}(t) \tau_{\widetilde{N}} G_{\widetilde{N}}(t)))-P_+(\tau_{\widetilde{N}}))W_{\widetilde{N}}(t)\\
&=& P_+(\frac{\frac{D_{\widetilde{N}}}{t+1}+ \tau_{\widetilde{N}}}{|\frac{D_{\widetilde{N}}}{t+1}+ \tau_{\widetilde{N}}|})
-P_+(\frac{\frac{D_{\widetilde{N}}}{t+1}- \tau_{\widetilde{N}} }{|\frac{D_{\widetilde{N}}}{t+1}- \tau_{\widetilde{N}}| })\\
&=&P_+(\frac{D_{\widetilde{N}}}{t+1}+ \tau_{\widetilde{N}})- P_+(\frac{D_{\widetilde{N}}}{t+1}- \tau_{\widetilde{N}}).
 	\end{eqnarray*}
 	This finishes the proof.
\end{proof}

\begin{lemma}\label{lemma equiv of F smooth and op}
We have $[F_{C^\infty,\widetilde{N}}]= [F'_{\widetilde{N}}]\in K_0(C^*_L(\widetilde{N})^\Gamma).$
\end{lemma}
\begin{proof}
It is direct from the comparing of the definition of $F_{C^\infty,\widetilde{N}}$ with the definition of $F'_{\widetilde{N}}$.
 \end{proof}

\section{Higher $\rho$ invariant}\label{sec higher rho}
% there are mainly two  different approaches to the higher rho invariant associated to homotopy equivalences for the signature operator

We are now ready to recall those two approaches to the higher $\rho$ invariant of the signature operator associated to the homotopy equivalence in the literature, and show that they are equivalent to each other.

\subsection{Defnition of $\rho_{op}$}\label{sec def of rho op}
In this subsection, we recall the definition of the higher $\rho$ invariant associated to the homotopy equivalence given by Piazza, Schick in \cite{PS16ThesurgeryexactsequenceK-theoryandthesignatureoperator}.

 Piazza, Schick defined the higher $\rho$ invariant of the signature operator associated to a homotopy equivalence of $n$ dimensional manifolds $f: M \to N$ by the Hilsum-Skandalis perturbation (see \cite{HS92Invarianceparhomotopiedelasignatureacoefficientsdansunfibrepresqueplat} for detailed reference). Let $\Gamma$ be the fundamental group of $M$ and $N$, $\widetilde{M}$ and $\widetilde{N}$ be the universal covering spaces  of $M$ and $N$ respectively. Let $L^2(\wedge^p ( \widetilde{M}))$ and $L^2(\wedge^p (\widetilde{N}))$ be Hilbert spaces of $L^2$ $p-$forms on $\widetilde M$ and $\widetilde N$ respectively. Set $d_{\widetilde M}$, $d_{\widetilde N}$ to be the $\Gamma$-invariant differentials and $\tau_{\widetilde M}$, $\tau_{\widetilde N}$ to be the chiral dualities. Let $\widetilde f:\widetilde M\to\widetilde N$ be the induced homotopy equivalence.

  Note that $\widetilde{f}: L^2(\wedge^p (\widetilde{N}))\to L^2(\wedge^p ( \widetilde{M}))$ is not a bounded linear operator in general. However, by an embedding procedure (\cite{HS92Invarianceparhomotopiedelasignatureacoefficientsdansunfibrepresqueplat}),  Hilsum and Skandalis showed that $\widetilde f$ induces a bounded operator $T_{\widetilde f}:  L^2(\wedge^p (\widetilde{N}))\to L^2(\wedge^p ( \widetilde{M}))$ in \cite{HS92Invarianceparhomotopiedelasignatureacoefficientsdansunfibrepresqueplat}. Let $D_{\widetilde M}$ be $d_{\widetilde M}+d_{\widetilde M}^*$ and $D_{\widetilde N}$ be $d_{\widetilde N}+d_{\widetilde N}^*$. Wahl  modified $T_{\widetilde f}$ to be  $\mathcal{T}_{\widetilde f}:= \phi_\epsilon (D_{\widetilde M}) T_{\widetilde f} \phi_\epsilon (D_{\widetilde N}), $ where $\phi:\mathbb{R}\to [0,1]$  is an even $C_c^\infty(\mathbb{R})$ function, such that $\text{Supp}(\phi) \subset [-\frac{1}{2}, \frac{1}{2}]$, $\phi(x)=1 $ for $x\in [-\frac{1}{4}, \frac{1}{4}]$, and $\phi_\epsilon (x)= \phi(\frac{x}{\epsilon})$ (\cite{W13Higherrho-invariantsandthesurgerystructureset}). We will not recall all the details of Wahl's construction of $\mathcal{T}_{\widetilde f}$, but only mention that it is a controlled chain homotopy equivalence with finite propagation from $(L^2(\wedge^* (\widetilde{N})), d_{\widetilde{N}}, \tau_{\widetilde{N}})$ to $(L^2(\wedge^* ( \widetilde{M})) , d_{\widetilde{M}}, \tau_{\widetilde{M}})$.

In the following, we denote by $\mathcal{T}_{\widetilde f}^*$ the adjoint of $\mathcal{T}_{\widetilde f}$ according to the inner product of  $L^2(\wedge^* ( \widetilde{M}))$ and $L^2(\wedge^* (\widetilde{N}))$.

The higher $\rho$ invariant of  Piazza and  Schick is then constructed as follows.

Let
\[
\tau=\left(
\begin{array}{cc}
\tau_{\widetilde N} & \\
 & -\tau_{\widetilde M}
\end{array}
\right)
\]
be the chiral duality operator of $\widetilde N \sqcup -\widetilde M$. A basic fact is that there exist bounded operators $y: \wedge^p(\widetilde N) \to \wedge^{p-1}(\widetilde N)$ with finite propagation (cf. \cite[Lemma 2.2]{W13Higherrho-invariantsandthesurgerystructureset}) such that
\[
 y':=-\tau y\tau =(-1)^{n+1} y,
\]
\[
1 -\tau_N \mathcal{T}_{\widetilde f}^* \tau_{\widetilde M} \mathcal{T}_{\widetilde f} = d_{\widetilde N} y + y d_{\widetilde N}.
\]

Let $\gamma \omega = (-1)^k \omega$ for $\omega\in  L^2(\wedge^k (\widetilde{N} \cup \widetilde{M}) )$. For $0\leq \alpha, \beta\leq 1, \ (\beta-1)\alpha=0$, set
\[
R_\beta(\widetilde f) = \left(
\begin{array}{cc}
1 & 0 \\
\beta \mathcal{T}_{\widetilde f} & 1
\end{array}
\right)\]
and
\[ L_{\alpha, \beta}({\widetilde f})=  \left(
\begin{array}{cc}
1- \beta^2 \tau_{\widetilde N} \mathcal{T}_{\widetilde f}^* \tau_{\widetilde M} \mathcal{T}_{\widetilde f}  & -(1-i\alpha \gamma y)\beta \tau_{\widetilde N} \mathcal{T}_{\widetilde f}^* \tau_{\widetilde M} \\
\beta  \mathcal{T}_{\widetilde f} (1+i\alpha \gamma y) & 1
\end{array}
\right).
\]
$L_{\alpha, \beta}({\widetilde f})$ is invertible when $\alpha$ is small enough since
\[
L_{\alpha, \beta}({\widetilde f})= R'_\beta({\widetilde f}) R_\beta({\widetilde f}) +  \left(
\begin{array}{cc}
0  & i\alpha \gamma y\beta \tau_N \mathcal{T}_{\widetilde f}^* \tau_M \\
\beta  \mathcal{T}_{\widetilde f} i\alpha \gamma y & 0
\end{array}
\right),
\]
where $ R'_\beta({\widetilde f})$ is the matrix
\[
\begin{pmatrix}
1 & -\beta \tau_{\widetilde N} \mathcal{T}_{\widetilde f}^* \tau_{\widetilde M} \\
0 & 1
\end{pmatrix}.
\]

By construction, the operator $\tau L_{\alpha, \beta}({\widetilde f})$ is self-adjoint.
Set $\mathcal{U}_{\alpha, \beta}({\widetilde f})= |\tau L_{\alpha, \beta}({\widetilde f})|^{\frac{1}{2}}$, and
\[
d_{\alpha, \beta }({\widetilde f}) = \mathcal{U}_{\alpha, \beta}({\widetilde f}) \left(
\begin{array}{cc}
d_{\widetilde N} &  -i\alpha \tau_{\widetilde N} \mathcal{T}_{\widetilde f}^* \tau_{\widetilde M}\gamma\\
0   &  d_{\widetilde M}
\end{array}
\right) \mathcal{U}_{\alpha, \beta}({\widetilde f})^{-1}.
\]
Then $d_{\alpha, \beta}(\widetilde f)$ anticommutes with the operator
\[
S_{\alpha, \beta}(\widetilde f)=  \mathcal{U}_{\alpha, \beta}({\widetilde f})\frac{\tau L_{\alpha,\beta}}{|\tau L_{\alpha,\beta}|} \mathcal{U}_{\alpha, \beta}({\widetilde f})^{-1}=\frac{\tau L_{\alpha,\beta}}{|\tau L_{\alpha,\beta}|}.
\]
Set
\[D_{\alpha, \beta}({\widetilde f})=d_{\alpha, \beta }({\widetilde f}) + d^*_{\alpha, \beta }({\widetilde f})=d_{\alpha, \beta }({\widetilde f})-S_{\alpha, \beta}({\widetilde f})d_{\alpha, \beta }({\widetilde f})S_{\alpha, \beta}({\widetilde f}). \]
By the theory of Hilsum and Skandalis in \cite{HS92Invarianceparhomotopiedelasignatureacoefficientsdansunfibrepresqueplat}, the operator $D_{\alpha, \beta}({\widetilde f})$ is invertible.

Recall
\[
\frac{D_{\alpha, \beta}({\widetilde f}) }{|D_{\alpha,\beta}({\widetilde f})| }
\]
can be decomposed as
\[
\left(\begin{array}{cc}
0 & U_+ \\
U_- & 0
\end{array}
\right)
\]
according to the $\pm 1$ eigenspaces of $S_{\alpha, \beta}({\widetilde f})$.
By \cite{W13Higherrho-invariantsandthesurgerystructureset} and \cite{PS16ThesurgeryexactsequenceK-theoryandthesignatureoperator}, we know that $\rho_{op}(f):=U_+$ defines a class in  $K_1(D^*(\widetilde{N})^\Gamma)$. We denote this class by $[\rho_{op}(f)]$.

For more details of the construction of  $D_{\alpha, \beta}({\widetilde f})$
and $S_{\alpha, \beta}({\widetilde f})$, see \cite{HS92Invarianceparhomotopiedelasignatureacoefficientsdansunfibrepresqueplat, PS07JNCG, W13Higherrho-invariantsandthesurgerystructureset}.  When no confusion is likely  to arise, we will denote  $D_{\alpha, \beta}({\widetilde f})$,
and $S_{\alpha, \beta}({\widetilde f})$ simply by $D_{\alpha, \beta}$ and $S_{\alpha, \beta}$ respectively.

We now describe  the image of $[\rho_{op}(f)]$ under $\text{Ind}_L$ by applying Roe's index formula (see \cite{R93CoarsecohomologyandindextheoryoncompleteRiemannianmanifolds}). As before, set $g(x)=\frac{x}{\sqrt{1+x^2}},\ h=\sqrt{1-g^2}$. Note that $D_{\alpha, \beta}$ is invertible. Set operators as
\[
	G_{\alpha, \beta}(t)=g(\frac{D_{\alpha, \beta}}{t}) -S_{\alpha, \beta } h(\frac{D_{\alpha, \beta}}{t}) , t\in [0,+\infty).
\]
Then the path $t\mapsto G_{\alpha, \beta}(t)$ is a symmetry in $D^*_L(\widetilde{N})^{\Gamma}$, and
\[G_{\alpha, \beta}S_{\alpha, \beta}+S_{\alpha, \beta } G_{\alpha, \beta} \in C^*_{L,0}(\widetilde{N})^\Gamma.\]
The following formal difference
\[
[P_+(-G_{\alpha, \beta }S_{\alpha, \beta } G_{\alpha, \beta } ) ]-[P_+(\tau_{\alpha, \beta })]
\]
then defines an element in $K_0(C_{L, 0}^*(\widetilde{N})^G).$
 We actually have
\[
\text{Ind}_L [\rho_{op}(f)]=[P'_+(-G_{\alpha, \beta }S_{\alpha, \beta } G_{\alpha, \beta } ) ]-[P_+(S_{\alpha, \beta })].
\]

At last, we remind the reader to notice that
\[
(L^2(\wedge^{*+1} ( \widetilde{N}))\oplus L^2(\wedge^* ( \widetilde{M})), d_{\alpha, \beta}, S_{\alpha, \beta})
\]
is an analytically controlled Hilbert Poincar\'e complex over $\widetilde{N}$. This is implied by the construction and computation in \cite{W13Higherrho-invariantsandthesurgerystructureset} and \cite{PS16ThesurgeryexactsequenceK-theoryandthesignatureoperator}.

\subsection{Definition of $\rho_{PL}$}\label{sec def of rho PL}
In this subsection, we recall the definition of the representative  of the higher $\rho$ invariant given by Weinberger, Xie and Yu in \cite{WXY18Additivityofhigherrhoinvariantsandnonrigidityoftopologicalmanifolds}. We will denote this representative by $\rho_{PL}(f)$. Originally, Weinberger, Xie and Yu used the simplicial chain complex to construct $\rho_{PL}(f)$. Equivalently, we will use the simplicial cochain complex instead.

Let $f: M \to N$ be a homotopy equivalence from the smooth manifold $M$ to the smooth manifold $N$.
Take a $\Gamma$-invariant triangulation for $\widetilde{M}$ and $ \widetilde{N}$. Then the corresponding $L^2$-completion gives the following two analytically controlled Hilbert Poincar\'e complexes
\[
 (L^2(E^*({\widetilde{N}})) ,b_{\widetilde{N}}, S_{\widetilde{N}} )\ \text{and} \ (L^2(E^*(\widetilde{M}))  ,b_{\widetilde{M}}, S_{\widetilde{M}} ).
\]
Let $\mathcal{F}$ be the simplicial map of ${\widetilde{f}}$. Then \cite[Lemma 4. 35]{WXY18Additivityofhigherrhoinvariantsandnonrigidityoftopologicalmanifolds} shows that
\[\mathcal{F}: (L^2(E^*({\widetilde{N}})) ,b_{\widetilde{N}}, S_{\widetilde{N}} ) \to (L^2(E^*(\widetilde{M}))  ,b_{\widetilde{M}}, S_{\widetilde{M}} )\]
is an analytically controlled homotopy equivalence.

Set
\[
b= \left(\begin{array}{cc}
b_{\widetilde{N}} & \\
 & b_{\widetilde{M}}
\end{array} \right)\ \text{and}\ S= \left(\begin{array}{cc}
S_ {\widetilde{N}}& \\
 & -S_{\widetilde{M}}
\end{array} \right).
\]
Recall that in Subsection \ref{subsec L2 simplicial and bounded complex}, we have defined the $K$-homology class of the signature operator of
\[( L^2(E^*(\widetilde{N}))  \oplus L^2(E_{*}(\widetilde{M})) ,b, S )\]
in $K_0(C^*_{L}(\widetilde N)^\Gamma),$ which is represented by a path of differences of projections. We will denote the class as $[F_{PL}],$ and the path of differences representing it as $F_{PL}.$ Note that $[F_{PL}]= [F_{PL,\widetilde N}]-[F_{PL,\widetilde M}].$

Then the higher $\rho$ invariant, $[\rho_{PL}(f)],$
is defined as the class in $K_0(C^*_{L,0}(\widetilde{M})^\Gamma)$  represented by the following path of differences of projections
\[
\rho_{PL}(f):=\left\{
\begin{array}{cc}
P[\mathcal{F}](3-t)-Q[\mathcal{F}](3-t) &  t\in [0,3], \\
F_{PL}(t-3) & t\in [3,\infty].	
\end{array}
\right.
\]
where $P[\mathcal{F}]$ and $Q[\mathcal{F}]$ are defined in Subsection \ref{subsec homo inva of the signature class}.

\subsection{Definition of $\rho_{C^\infty}$}\label{sec def of rho smooth}
In this subsection, we  give the definition of the representative of the higher $\rho$ invariant by considering  the deRham complex. We will denote this class by $\rho_{C^\infty}$. This definition has been examined in \cite{JL18additivehigherrhoinvariantforstructuregroupindifferentialpointofview}.

Consider analytically controlled Hilbert Poincar\'e complexes
\[
( L^2(\wedge^{*} ( \widetilde{N})),
d_{\widetilde{N}} , \tau_{\widetilde{N}})\  \text{and} \  ( L^2(\wedge^{*} ( \widetilde{M})),
d_{\widetilde{M}} , \tau_{\widetilde{M}}).
\]
As shown in Subsection \ref{sec def of rho op},
\[\mathcal{T}_{\widetilde f}:
( L^2(\wedge^{*} ( \widetilde{N})),
d_{\widetilde{N}} , \tau_{\widetilde{N}})\to  ( L^2(\wedge^{*} ( \widetilde{M})),
d_{\widetilde{M}} , \tau_{\widetilde{M}}).
\]
is an analytically controlled homotopy equivalence.

We set
\[d=\left(
\begin{array}{cc}
d_{\widetilde{N}} & 0 \\
0   & d_{\widetilde{M}}
\end{array}
\right)\ \text{and}\ \tau= \left(
\begin{array}{cc}
\tau_{\widetilde{N}} & 0 \\
0   & -\tau_{\widetilde{M}}
\end{array}
\right). \]
Consider the $\Gamma$-equivariant analytically controlled Hilbert Poincar\'e complex over $\widetilde{N}$
\[
( L^2(\wedge^{*} ( \widetilde{N}))\oplus L^2(\wedge^{*} ( \widetilde{M})),
d , \tau).
\]
We denote the  path of differences of projections representing the $K$-homology class of the signature operator of this complex, defined in Subsection \ref{subsec L 2 form and unbounded com}, by $F_{C^\infty},$ and the $K$-homology class of the signature operator by $[F_{C^\infty}].$ Note that by definition, $[F_{C^\infty}]$  equals $[F_{C^\infty, \widetilde N}]-[F_{C^\infty,\widetilde M}]$.

Then the higher $\rho$ invariant associated to $f$, $[\rho_{C^\infty}(f)]$
is the class  in $K_0(C^*_{L,0}(\widetilde N)^\Gamma)$ represented by the following path of differences
\[
\rho_{C^\infty}(f):=\left\{
\begin{array}{cc}
P[\mathcal{T}_{\widetilde f}](3-t)-Q[\mathcal{T}_{\widetilde f}](3-t) &  t\in [0,3],\\
F_{C^\infty}(t-3) & t\in [3,\infty].	
\end{array}
\right.
\]

\subsection{Equivalence between $\rho_{PL}$ and $\rho_{C^\infty}$}\label{sec equiv rhp PL and smoth}
In this subsection, we prove the equivalence between $\rho_{PL}(f)$ and $\rho_{C^\infty}(f)$ for homotopy equivalence $f: M\to N$.

\begin{theorem}\label{theo equiv of rho PL and smooth}
With the notations as above, we have
\[[\rho_{PL}(f)]= [\rho_{C^\infty}(f)  ] \in K_0(C^*_L(\widetilde{N})^\Gamma).\]
\end{theorem}

\begin{proof}
Recall the construction of $\rho_{PL}(f)$ and $\rho_{C^\infty}(f)$.
Again we consider the homotopy equivalence  from the deRham complex to the simplicial cochain complex
\[H:(L^2(E^*(\widetilde{N}))\oplus L^2( E_{*}(\widetilde{M})), b, S) \to (L^2(\wedge^* ( \widetilde{N})  ) \oplus L^2(\wedge^{*}  (\widetilde{N})  )  , d, \tau).\]
Set
\[
S_{PL}(t)=\left\{\begin{array}{cc}
	\left(
\begin{array}{cc}
(1-t)S_{\widetilde N} + t \mathcal{F}^* S_{\widetilde M} \mathcal{F} & 0\\
0 &  -S_{\widetilde M}
\end{array}
\right) & t\in [0,1],\\
\left(
\begin{array}{cc}
\cos (\frac{\pi}{2}(t-1)) \mathcal{F}^* S_{\widetilde M} \mathcal{F} & \sin (\frac{\pi}{2}(t-1)) \mathcal{F}^* S_{\widetilde M}  \\
\sin (\frac{\pi}{2}(t-1))   S_{\widetilde M} \mathcal{F}  & - \cos (\frac{\pi}{2}(t-1)) S_{\widetilde M}
\end{array}
\right) & t\in [1,2] ,\\
 \left(
\begin{array}{cc}
	0 &  e^{i\pi (t-2)}S_{\widetilde M} \mathcal{F}\\
	e^{-i\pi (t-2)}\mathcal{F}^*S_{\widetilde M}  & 0
\end{array}
\right) & t\in [2,3],
\end{array}
\right.
\]
and
\[
\tau_{C^\infty}(t)=\left\{\begin{array}{cc}
	\left(
\begin{array}{cc}
(1-t)\tau_{\widetilde N} + t \mathcal{T}_{\widetilde f}^* \tau_{\widetilde M} \mathcal{T}_{\widetilde f} & 0\\
0 &  -\tau_{\widetilde M}
\end{array}
\right) & t\in [0,1],\\
\left(
\begin{array}{cc}
\cos (\frac{\pi}{2}(t-1)) \mathcal{T}_{\widetilde f}^* \tau_{\widetilde M} \mathcal{T}_{\widetilde f} & \sin (\frac{\pi}{2}(t-1)) \mathcal{T}_{\widetilde f}^* \tau_{\widetilde M}  \\
\sin (\frac{\pi}{2}(t-1))   \tau_{\widetilde M} \mathcal{T}_{\widetilde f}   & - \cos (\frac{\pi}{2}(t-1)) \tau_{\widetilde M}
\end{array}
\right) & t\in [1,2],\\
 \left(
\begin{array}{cc}
	0 &  e^{i\pi (t-2)}\tau_{\widetilde M} \mathcal{T}_{\widetilde f}\\
	e^{-i\pi (t-2)}\mathcal{T}_{\widetilde f}^*\tau_{\widetilde M}   & 0
\end{array}
\right) & t\in [2,3].
\end{array}
\right.
\]
Then the canonical homotopy equivalence $H$, between the deRham complex and the simplicial cochain complex, induces an analytically controlled homotopy equivalence
\begin{eqnarray*}
H&:& (L^2(E^*(\widetilde{N}))\oplus L^2( E_{*}(\widetilde{M})), b, S_{PL}(t))\\
&&\to (L^2(\wedge^* ( \widetilde{N})  ) \oplus L^2(\wedge^{*}  (\widetilde{M} ) )  , d, \tau_{C^\infty}(t))
\end{eqnarray*}
for any $t\in [0,3].$
Then this lemma follows from Lemma \ref{lemma equiv of F PL and smooth}.
\end{proof}

\subsection{Equivalence between $\rho_{op}$ and $\rho_{C^\infty}$}\label{sec equiv rho op rho smooth}
In this subsection, we prove that for any homotopy equivalence $f: M \to N$, $[\rho_{C^\infty}(f) ]$ equals $\text{Ind}_L[\rho_{op}(f)]$ in $K_0(C^*_{L,0}(\widetilde{N})^\Gamma)$.

The following lemma describes $\text{Ind}_L(\rho_{op}(f) ).$
\begin{lemma}\label{lemma depicts rho op after comp}
The class $\text{Ind}_L[\rho_{op}(f)] \in K_0(C^*_{L,0}(\widetilde{N})^\Gamma)$ can be represented by the following path of differences:
\[
\left\{
\begin{array}{cc}
P_+(D_{\alpha,  \beta} + tS_{\alpha, \beta}) - P_+ (D_{\alpha, \beta} - tS_{\alpha, \beta}) & t\in [0,1],\\
P_+(D_{(2-t)\alpha, \beta} +  S_{(2-t)\alpha, \beta}) - P_+ (D_{(2-t)\alpha,  \beta} - S_{(2-t)\alpha, \beta}) & t\in [1,2],\\
P_+(D_{0, (3-t)\beta} + S_{0, (3-t)\beta}) - P_+ (D_{0,(3-t)\beta} - S_{0, (3-t)\beta}) & t\in [2,3],\\
F_{C^\infty}(t-3) & t\in [3, \infty).
\end{array}
\right.
\]
where $F_{C^\infty}$ is defined in Subsection \ref{subsec L 2 form and unbounded com}.
\end{lemma}
\begin{proof}

Let $\widetilde{S}_{\alpha,\beta}(t)$ be
\[
\left\{
\begin{array}{cc}
	S_{\alpha, \beta } & t\in [0,1],\\
	S_{(2-t)\alpha, \beta} & t\in [1,2],\\
	S_{0, (3-t)\beta}  & t\in [2,3],\\
	\tau & t\in [3, \infty).
\end{array}
\right.
\]
Set $g_{t}(x)= \frac{x}{\sqrt{x^2+t^2}}$, $h_t(x)=\sqrt{1-g_t^2(x)}$.
 We then  consider operators
 \[
 \widetilde{G}_{\alpha, \beta}(t)=\left\{
\begin{array}{cc}
	g_t(D_{\alpha, \beta}) -\widetilde{S}_{\alpha,\beta}(t) h_t(D_{\alpha, \beta}) & t\in [0,1],\\
	g_1 (D_{(2-t)\alpha, \beta}  ) - \widetilde{S}_{(2-t)\alpha,\beta}(t)h_1(D_{(2-t)\alpha, \beta}) & t\in [1,2],\\
	g_1 (D_{0, (3-t)\beta} ) -  \widetilde{S}_{\alpha,(3-t)\beta}(t) h_1(D_{0, (3-t)\beta})  & t\in [2,3],\\
	g_1 (\frac{D}{t-2})-\tau h_1(\frac{D}{t-2}) & t\in [3, \infty).
\end{array}
\right.
 \]

By the construction of $[\rho_{op}(f)]$, we know that $\text{Ind}_L[\rho_{op}(f)]$, which is defined in Subsection \ref{sec def of rho op}, equals the class
\[
[P_+(- \widetilde{G}_{\alpha, \beta}\widetilde{S}_{\alpha, \beta} \widetilde{G}_{\alpha, \beta})]-[P_+( \widetilde{S}_{\alpha, \beta})].
\]

We first consider $t\in [0, 3]$. Note that
\[(\widetilde{G}_{\alpha, \beta}(t) - \widetilde{S}_{\alpha, \beta}(t)  )^2\]
is invertible for any $t\in [0, 3]$.
Thus
\[I - \widetilde{S}_{\alpha, \beta}(t)\widetilde{G}_{\alpha, \beta}(t) =   \widetilde{S}_{\alpha, \beta}(t) (  \widetilde{S}_{\alpha, \beta}(t) - \widetilde{G}_{\alpha, \beta}(t))\]
is invertible, which implies that $1$ does not belong to the spectrum of $\widetilde{S}_{\alpha, \beta}(t)\widetilde{G}_{\alpha, \beta}(t)$. We form a square root
\[
V(t)= \sqrt{\widetilde{S}_{\alpha, \beta}(t)\widetilde{G}_{\alpha, \beta}(t)}.
\]
Let $W(t)$ be $V(t)\widetilde{G}_{\alpha, \beta}(t) $. Direct computation shows that
\begin{eqnarray*}
W(t)^*\widetilde{S}_{\alpha, \beta}(t)W(t)&=&  \widetilde{G}_{\alpha, \beta}(t),
\\	
W(t)^*\widetilde{G}_{\alpha, \beta}(t)
W(t)&=&  \widetilde{S}_{\alpha, \beta}(t).
\end{eqnarray*}
Thus there is
\begin{eqnarray*}
&&	W(t)^* [P_+(- \widetilde{G}_{\alpha, \beta}(t) \widetilde{S}_{\alpha, \beta}(t) \widetilde{G}_{\alpha, \beta}(t) )]W(t)-W^*(t)[ P_+(\widetilde{S}_{\alpha, \beta}(t))]   W(t)^* \\
	& =& [P_+(-\widetilde{S}_{\alpha, \beta}(t)  \widetilde{G}_{\alpha, \beta}(t) \widetilde{S}_{\alpha, \beta}(t) )]-[P'_+( \widetilde{G}_{\alpha, \beta}(t)) ].
\end{eqnarray*}
The latter formal difference equals the class represented by the path of formal differences
\[
\left\{
\begin{array}{cc}
P_+(\frac{D_{\alpha,  \beta} + tS_{\alpha, \beta}}{|D_{\alpha,  \beta} + tS_{\alpha, \beta}|} )- P_+(\frac{D_{\alpha,  \beta} - tS_{\alpha, \beta}}{|D_{\alpha,  \beta} - tS_{\alpha, \beta}|} ) & t\in [0,1],\\

P_+(\frac{D_{(2-t)\alpha, \beta} +  S_{(2-t)\alpha, \beta}}{|D_{(2-t)\alpha, \beta} +  S_{(2-t)\alpha, \beta}|}) - P_+(\frac{D_{(2-t)\alpha, \beta} -  S_{(2-t)\alpha, \beta}}{|D_{(2-t)\alpha, \beta} - S_{(2-t)\alpha, \beta}|}) & t\in [1,2],\\

P_+(\frac{D_{0, (3-t)\beta} + S_{0, (3-t)\beta}}{|D_{0, (3-t)\beta} +S_{0, (3-t)\beta}|}) ]- [P_+(\frac{D_{0, (3-t)\beta} - S_{0, (3-t)\beta}}{|D_{0, (3-t)\beta} - S_{0, (3-t)\beta}|}) & t\in [2,3].
\end{array}
\right.
\]
Then the lemma follows directly from Lemma \ref{lemma equivalent def of F op} and Lemma  \ref{lemma equiv of F smooth and op}.
\end{proof}

The following lemma follows directly from the proof of \cite[Theorem 5.2]{chenliuyu}.
\begin{lemma}\label{lemma rho trivial of diffeo}
For $I: M \to M$, both $[\rho_{op}(I)]\in K_*(C^*_{L,0}(\widetilde{N})^\Gamma)$ and $[\rho_{C^\infty}(I)]\in K_*(C^*_{L,0}(\widetilde{N})^\Gamma)$ are trivial.
\end{lemma}

Consider $  I : M \to M $ and $f: M \to N  $. For convenience, we will denote $f_* [\rho_{op}(I)]\in K_1(D^*(\widetilde{N})^\Gamma)$ merely by $ [\rho_{op}(I)] $, and regard $\rho_{op}(I)$ as an element in $ D^*(\widetilde{N})^\Gamma $.

 Let $L^2(\wedge^* (I: \widetilde M\to \widetilde M))$ be $L^2(\wedge^{*+1} (\widetilde{M} ))\oplus L^2(\wedge^* (\widetilde{M} )),$ and $L^2(\wedge^* (\widetilde f: \widetilde M\to\widetilde N))$ be $L^2(\wedge^{*+1} (\widetilde{N} ))\oplus L^2(\wedge^* (\widetilde{M}) ).$ Then
\[ (L^2(\wedge^* (I: \widetilde M\to \widetilde M)) , d_{\alpha, \beta}(I),  S_{\alpha, \beta}(I) )\] and
\[ (L^2(\wedge^* (\widetilde f: \widetilde M\to\widetilde N)) , d_{\alpha,\beta}(\widetilde{f}),  S_{\alpha, \beta}(\widetilde{f})) \]
are homotopy equivalent analytically controlled Hilbert Poincar\'e complexes, whereas the homotopy equivalence can be given by $\widetilde f \sqcup I$. Following the argument in \cite{W13Higherrho-invariantsandthesurgerystructureset}, one can directly show that there is the corresponding $\mathcal{T}_{(\widetilde f \sqcup I)} $. By the theory of the analytically controlled homotopy equivalence of analytically controlled Hilbert Poincar\'e complexes, we have
\begin{lemma}\label{lemma acpc path for I plus f 1}
Set $S_{\alpha, \beta , I\sqcup \widetilde f}(t)$ as
\[
\begin{pmatrix}
(1-t) S_{\alpha, \beta}(\widetilde f) + t\mathcal{T}_{(\widetilde f \sqcup I)}^* S_{\alpha, \beta}(I)\mathcal{T}_{(\widetilde f \sqcup I)}  & 0 \\
0   &   - S_{\alpha, \beta}(I)
\end{pmatrix}
\]
when $t\in [0,1],$ as
\[
\begin{pmatrix}
\cos((t-1)\frac{\pi}{2} ) \mathcal{T}_{(\widetilde f \sqcup I)}^* S_{\alpha, \beta}(I)\mathcal{T}_{(\widetilde f \sqcup I)}  & \sin((t-1)\frac{\pi}{2} ) \mathcal{T}_{(\widetilde f \sqcup I)}^*S_{\alpha, \beta}(I) \\
\sin((t-1)\frac{\pi}{2} ) S_{\alpha, \beta}(I)\mathcal{T}_{(\widetilde f \sqcup I)} &   -\cos((t-1)\frac{\pi}{2} ) S_{\alpha, \beta}(I)
\end{pmatrix}\]
for $t\in [1,2]$, and as
\[
\begin{pmatrix}
0  & e^{i(t-2)\pi}  \mathcal{T}_{(\widetilde f \sqcup I)}^* S_{\alpha, \beta}(I) \\
e^{-i(t-2)\pi} S_{\alpha, \beta}(I)\mathcal{T}_{(\widetilde f \sqcup I)}  &  0
\end{pmatrix}
\]
as $t\in [2,3]$. Then for any $t\in [0,3]$,
\[
( L^2(\wedge^* (I: \widetilde M\to \widetilde M))\oplus L^2(\wedge^* (\widetilde f: \widetilde M\to\widetilde N) ) ,  d_{\alpha,\beta}(I)\oplus  d_{\alpha,\beta}(\widetilde f), S_{\alpha, \beta }(t))
\]
is a $\Gamma$-equivariant analytically controlled Hilbert Poincar\'e complex over $\widetilde{N}$.
\end{lemma}
In the following of this section, we denote $d_{\alpha,\beta}(I)\oplus  d_{\alpha,\beta}(\widetilde f)$ by $d_{\alpha,\beta}$, and $D_{\alpha,\beta}(I)\oplus  D_{\alpha,\beta}(\widetilde f)$ by $D_{\alpha,\beta}$.

Note that
\[
S_{0,0, \widetilde f \sqcup I}(0)=\begin{pmatrix}
\tau_N & & & \\
& -\tau_M & & \\
& & -\tau_M & \\
& & & \tau_M
\end{pmatrix}.
\]

Let
\[
p(t)=\left\{
\begin{array}{cc}
P_+(D_{\alpha, \beta} + tS_{\alpha, \beta ,  \widetilde f \sqcup I }(2) ) & t\in [0,\frac{1}{3}],\\
P_+(D_{\alpha, \beta} + tS_{\alpha, \beta , \widetilde f \sqcup I}(3-3t) )&  t\in [\frac{1}{3},\frac{2}{3}],\\
P_+(D_{\alpha, \beta} + tS_{\alpha, \beta , \widetilde f \sqcup I}(3-3t) )& t\in [\frac{2}{3},1],\\
P_+(D_{(2-t)\alpha, \beta} + S_{(2-t)\alpha, \beta, \widetilde f \sqcup I}(0)) &  t\in [1,2],\\
P_+(D_{0, (3-t)\beta} + S_{0, (3-t)\beta, \widetilde f \sqcup I}(0)) &  t\in [2,3].
\end{array}
\right.
\]
Let
\[
q(t)=
\left\{
\begin{array}{cc}
 P_+ (D_{\alpha, \beta}- tS_{\alpha, \beta , \widetilde f \sqcup I}(3-3t)) & t\in [0,\frac{1}{3}],\\
P_+ (D_{\alpha, \beta} - tS_{\alpha, \beta , \widetilde f \sqcup I}(3-3t))&  t\in [\frac{1}{3},\frac{2}{3}],\\
P_+ (D_{\alpha, \beta} - tS_{\alpha, \beta , \widetilde f \sqcup I}(3-3t))& t\in [\frac{2}{3},1],\\
 P_+ (D_{(2-t)\alpha,  \beta} - S_{(2-t)\alpha, \beta, \widetilde f \sqcup I}(0)) &  t\in [1,2],\\
 P_+ (D_{0, (3-t)\beta} - S_{0, (3-t)\beta, \widetilde f \sqcup I}(0))&  t\in [2,3].
\end{array}
\right.
\]
Then by Lemma \ref{lemma depicts rho op after comp} and Lemma \ref{lemma acpc path for I plus f 1}, the class $\text{Ind}_L[\rho_{op}(f) ] -\text{Ind}_L[\rho_{op}(I) ]\in K_0(C^*_{L,0}(\widetilde{N})^\Gamma)$ can be represented by the path of  differences:
\[
\left\{
\begin{array}{cc}
p(t)-q(t) & t\in [0,3],\\
F_{C^\infty}(t-3) & t\in [3, \infty).
\end{array}
\right.
\]

\begin{corollary}\label{coro depicts rho op of I plus f}
Let
\[
p'(t)=\left\{
\begin{array}{cc}
P_+(D_{\alpha, \beta} + S_{\alpha, \beta , \widetilde f \sqcup I}(2) ) & t\in [0,\frac{1}{3}],\\
P_+(D_{\alpha, \beta} + S_{\alpha, \beta , \widetilde f \sqcup I}(3-3t) )&  t\in [\frac{1}{3},1],\\
P_+(D_{(2-t)\alpha, \beta} +  S_{(2-t)\alpha, \beta , \widetilde f \sqcup I}(0))& t\in [1,2],\\
P_+(D_{0, (3-t)\beta} + S_{0, (3-t)\beta, \widetilde f \sqcup I}(0)) &  t\in [2,3].\\
\end{array}
\right.
\]
Let
\[
q'(t)=\left\{
\begin{array}{cc}
 P_+ (D_{\alpha, \beta}- S_{\alpha, \beta , \widetilde f \sqcup I}(3-3t)) & t\in [0,\frac{1}{3}],\\
 P_+ (D_{\alpha, \beta} - S_{\alpha, \beta , \widetilde f \sqcup I}(3-3t))&  t\in [\frac{1}{3},1],\\
P_+ (D_{(2-t)\alpha,  \beta} - S_{(2-t)\alpha, \beta , \widetilde f \sqcup I}(0)) & t\in [1,2],\\
 P_+ (D_{0, (3-t)\beta} - S_{0, (3-t)\beta, \widetilde f \sqcup I}(0)  &  t\in [2,3].\\
\end{array}
\right.
\]
Then the class $\text{Ind}_L[\rho_{op}(f) ] -\text{Ind}_L[\rho_{op}(I) ]\in K_0(C^*_{L,0}(\widetilde{N})^\Gamma)$ can be represented by path of  differences
\[
\left\{
\begin{array}{cc}
p'(t)-q'(t),& t\in [0,3],\\
F_{C^\infty}(t-3) & t\in [3, \infty).
\end{array}
\right.
\]
\end{corollary}

We then can prove the main theorem of this section:
\begin{theorem}\label{theo equivalence of rho op and smooth}
With notations as above it holds
\[
\text{Ind}_L[\rho_{op}(f)] = [\rho_{C^\infty} (f) ]\in K_0(C_L^*(\widetilde{N})^\Gamma).
\]
\end{theorem}
\begin{proof}
For $\alpha=\beta=0$, one can directly see that the representative elements in Corollary \ref{coro depicts rho op of I plus f} represents the class
\[
[\rho_{C^\infty} (f)] - [\rho_{C^\infty}(I)].
\]
Thus we have
\begin{eqnarray*}
&&\text{Ind}_L [\rho_{op}(f)]\\
&=&\text{Ind}_L [\rho_{op}(f)]-\text{Ind}_L[\rho_{op}(I)] \\
&=& [\rho_{C^\infty} (f)] - [\rho_{C^\infty}(I)]\\
&=&[\rho_{C^\infty} (f)].
\end{eqnarray*}
\end{proof}

\end{document}